\documentclass[12pt]{article}
\usepackage{amsfonts}
\usepackage{amssymb}
\usepackage{amsmath}
\usepackage{fullpage}
\usepackage{xy}
\xyoption{all}

\begin{document}



\newtheorem{thm}{Theorem}[section]
\newtheorem{prop}[thm]{Proposition}
\newtheorem{lem}[thm]{Lemma}
\newtheorem{df}[thm]{Definition}
\newtheorem{cor}[thm]{Corollary}
\newtheorem{rem}[thm]{Remark}
\newtheorem{ex}[thm]{Example}
\newenvironment{proof}{\medskip
\noindent {\bf Proof.}}{\hfill \rule{.5em}{1em}\mbox{}\bigskip}

\def\eps{\varepsilon}

\def\GlR#1{\mbox{\it Gl}(#1,\R)}
\def\glR#1{{\frak gl}(#1,\R)}
\def\GlC#1{\mbox{\it Gl}(#1,\C)}
\def\glC#1{{\frak gl}(#1,\C)}
\def\Gl#1{\mbox{\it Gl}(#1)}

\def\ov{\overline}
\def\ot{\otimes}
\def\und{\underline}
\def\w{\wedge}
\def\ra{\rightarrow}
\def\lra{\longrightarrow}

\def\rk{\mbox{\rm rk}}
\def\mod{\mbox{ mod }}

\newcommand{\al}{\alpha}
\newcommand{\be}{\beta}
\newcommand{\ga}{\gamma}
\newcommand{\la}{\lambda}
\newcommand{\om}{\omega}
\newcommand{\Om}{\Omega}
\renewcommand{\th}{\theta}
\newcommand{\Th}{\Theta}

\newcommand{\lb}{\langle}
\newcommand{\rb}{\rangle}

\def\pair#1#2{\left<#1,#2\right>}
\def\big#1{\displaystyle{#1}}

\def\N{{\Bbb N}}
\def\Z{{\Bbb Z}}
\def\R{{\Bbb R}}
\def\C{{\Bbb C}}
\def\F{{\Bbb F}}
\def\CP{{\Bbb C} {\Bbb P}}
\def\HP{{\Bbb H} {\Bbb P}}
\def\P{{\Bbb P}}
\def\Q{{\Bbb Q}}
\def\H{{\Bbb H}}
\def\K{{\Bbb K}}

\def\FV{{\frak F}_V}

\def\so{{\frak {so}}}
\def\co{{\frak {co}}}
\def\sl{{\frak {sl}}}
\def\su{{\frak {su}}}
\def\sp{{\frak {sp}}}
\def\asp{{\frak {asp}}}
\def\csp{{\frak {csp}}}
\def\spin{{\frak {spin}}}
\def\stab{{\frak {stab}}}
\def\g{{\frak g}}
\def\a{{\frak a}}
\def\h{{\frak h}}
\def\s{{\frak s}}
\def\k{{\frak k}}
\def\l{{\frak l}}
\def\m{{\frak m}}
\def\n{{\frak n}}
\def\t{{\frak t}}
\def\u{{\frak u}}
\def\z{{\frak z}}
\def\r{{\frak {rad}}}
\def\p{{\frak p}}
\def\L{{\frak L}}
\def\X{{\frak X}}
\def\gl{{\frak {gl}}}
\def\hol{{\frak {hol}}}
\renewcommand{\frak}{\mathfrak}
\renewcommand{\Bbb}{\mathbb}
\def\srk{\mbox{\rm srk}}

\def\hook{\mbox{}\begin{picture}(10,10)\put(1,0){\line(1,0){7}}
  \put(8,0){\line(0,1){7}}\end{picture}\mbox{}}
\newcommand{\cyclic}{\mathop{\kern0.9ex{{+}\kern-2.2ex\raise-.29ex%
      \hbox{\Large\hbox{$\circlearrowright$}}}}\limits}

\def\be{\begin{equation}}
\def\ee{\end{equation}}
\def\bi{\begin{enumerate}}
\def\ei{\end{enumerate}}
\def\ba{\begin{array}}
\def\ea{\end{array}}
\def\bea{\begin{eqnarray}}
\def\eea{\end{eqnarray}}
\def\ben{\begin{enumerate}}
\def\een{\end{enumerate}}
\def\linebr{\linebreak}

\title{On the Solvability of the Transvection group of Extrinsic Symplectic Symmetric Spaces}

\author{Lorenz J. Schwachh\"ofer}

\date{ }

\maketitle

\begin{abstract}
Let $M$ be a symplectic symmetric space, and let $\imath: M \ra V$ be an extrinsic symplectic symmetric immersion in the sense of \cite{KS}, i.e., $(V, \Om)$ is a symplectic vector space and $\imath$ is an injective symplectic immersion such that for each point $p \in M$, the geodesic symmetry in $p$ is compatible with the reflection in the affine normal space at $\imath(p)$.

We show that the existence of such an immersion implies that the transvection group of $M$ is solvable.

\end{abstract}
\section{Introduction}

Ever since their introduction by \'E. Cartan (\cite{Car}), symmetric spaces have been studied 
intensely from various viewpoints.
In \cite{Ferus}, D. Ferus introduced the notion of an {\em extrinsic symmetric space} which is 
a Riemannian symmetric space admitting an embedding into an Euclidean vector space 
such that the geodesic reflection at each $p \in M$ is the restriction of the reflection in the 
affine normal space of $M$ in $p$. In fact, Ferus gave a classification of all Riemannian extrinsic 
symmetric spaces, showing that most Riemannian symmetric spaces admit such an embedding.

Evidently, the concept of extrinsic symmetric spaces may be generalized to other classes as well. 
See \cite{Kath, Kim} for results on extrinsic pseudo-Riemannian symmetric spaces and \cite{GS} 
for extrinsic CR-symmetric spaces. In both cases, many examples of such embeddings are given. In fact, it is shown in \cite{Ferus, Kath}Êthat there is a one-to-one correspondence between (pseu\-do-)Rie\-mannian extrinsic symmetric spaces and symmetric extensions of the Lie algebra of the transvection group together with a derivation on this extension satisfying certain conditions. Here, the transvection group denotes the group generated by the geodesic reflections. This correspondence is the key idea of the Ferus' classification in the Riemannian case (\cite{Ferus}); in the pseu\-do-Rie\-mannian case, it yields large classes of examples (\cite{Kath}), although not a complete classification. 

In \cite{CGRS}, {\em extrinsic symplectic symmetric spaces} (e.s.s.s.) were considered for the 
first time, and their basic algebraic and geometric properties were described. In \cite{KS}, we slightly extended this concept to {\em extrinsic symplectic symmetric immersions} (e.s.s.i.), i.e., injective symplectic immersions $\imath: M \ra V$ of the symplectic symmetric space $M$ into a symplectic vector space $(V, \Om)$ with the property that
\[
\sigma_{N_p} \circ \imath = \imath \circ \sigma_p,
\]
where $\sigma_{N_p}: V \ra V$ is the reflection in the affine normal space $N_p = (d\imath_p(T_pM))^\perp$ and $\sigma_p: M \ra M$ is the geodesic reflection in $p \in M$.

In contrast to the aforementioned cases of (pseudo-)Riemannian extrinsic symmetric spaces, there are few known classes of examples of e.s.s.i. (cf. \cite{CGRS, R}). All known examples have a $3$-step nilpotent transvection group. One reason for the lack of further examples might be that there is no description of e.s.s.s. via symmetric extensions of the transvection Lie algebra analogous to the (pseudo-)Riemannian case. However, it is the aim of this article to present a more conceptual explanation for the difficulty of finding new examples. Namely, we shall prove the following result.

\

\noindent
{\bf Main Theorem} {\em Let $\imath: M \ra (V, \Om)$ be an extrinsic symplectic symmetric immersion of the symmetric space $M$ into the symplectic vector space $V$. Then the transvection group of $M$ is solvable.}

\

This paper is structured as follows. In section \ref{sec:prelim}, we first recall from \cite{CGRS} and \cite{KS}Êhow the condition of the existence of an e.s.s.i. can be equivalently formulated on the Lie algebra level. We also recall some basic properties of representations of the Lie algebra $\sl(2, \C)$, and on the general structure theory of Levi algebras.

In section \ref{sec:rank-one} we show that if there exists an e.s.s.i. of a symplectic symmetric space whose transvection group is not solvable, then there also must be an e.s.s.i. of a subspace whose transvection group has a grading and a semi-simple part of rank one. Finally, in section \ref{sec:solvable} we show that a symplectic symmetric space whose transvection group has such a grading and a semi-simple part of rank one cannot have an e.s.s.i., which completes the proof.

\section{Preliminaries} \label{sec:prelim}

\subsection{Symplectic realizations} \label{sec:sympl-real}

A {\em symmetric pair} is a pair $(\g, d\sigma)$ consisting of a Lie algebra $\g$ and an involution $d\sigma: \g \ra \g$ with $d\sigma^2 = Id$, or, equivalently, a Lie algebra $\g$ with a decomposition
\be \label{eq:symmetric-brackets}
\g = \k \oplus \p, \mbox{Êsuch that } [\k, \k] \subset \k, [\k, \p] \subset \p, [\p, \p] \subset \k.
\ee
Here, $\k$ and $\p$ are the $\pm1$-eigenspaces of $d\sigma$. We call such a symmetric pair {\em transvective} if $[\p, \p] = \k$, so that $\g$ is generated by $\p$. A symmetric pair is called {\em symplectic} if there is an $ad_\k$-invariant non-degenerate $2$-form $\om \in \Lambda^2 \p$.

Let $(V, \Om)$ be a symplectic vector space, and consider a symplectic orthogonal splitting
\[
V = V_1 \oplus V_2.
\]
Then the conjugation in $\sp(V, \Om)$ by the reflection $\sigma_0: V \ra V$ with $\sigma_0|_{V_i} := (-1)^i Id$ is a Lie algebra involution, which induces the symmetric space decomposition
\be \label{eq:define tilde k/p}
\sp(V, \Om) = \tilde \k \oplus \tilde \p, \mbox{ where } \ba{l} \tilde \k = \{ x \in \sp(V, \Om) \mid x V_i \subset V_i\} \mbox{ and }\\ \tilde \p = \{ x \in \sp(V, \Om) \mid x V_i \subset V_{i+1}\}, \mbox{ taking $i \mod 2$.}\ea
\ee

We recall from \cite{KS} the following definition.

\begin{df} \label{def:realize-algebra}
Let $\g = \k \oplus \p$ be a symmetric pair and $(V, \Om)$ a symplectic vector space. We call a Lie algebra homomorphism
\[
d\imath: \g \longrightarrow \asp(V, \Om), \mbox{\hspace{1cm}} d\imath(x) := \Lambda(x) + v(x) \in \sp(V, \Om) \oplus V
\]
an {\em extrinsic symplectic realization of $\g$} if there is a symplectic orthogonal decomposition $V = V_1 \oplus V_2$ such that
\bi
\item
$\k = \ker(v)$ and $v: \p \ra V_1$ is a linear isomorphism,
\item
We have $\Lambda \circ d\sigma = Ad_{\sigma_0} \circ \Lambda$ for the involutions $d\sigma: \g \ra \g$ and $Ad_{\sigma_0}: \sp(V, \Om) \ra \sp(V, \Om)$, or, equivalently, with the splitting (\ref{eq:define tilde k/p}) we have
\[
\Lambda(\k) \subset \tilde \k \mbox{ and } \Lambda(\p) \subset \tilde \p.
\]
\ei
\end{df}

Evidently, $\om := v^*(\Om) \in \Lambda^2 \p$ induces a $ad_\k$-invariant symplectic form, so that $\g$ becomes a symplectic symmetric pair. The correspondence between symmetric pairs and symmetric spaces is well understood. Namely, for each symplectic pair $(\g, d\sigma)$ there are symmetric spaces $G/K$ where $K \subset G$ are Lie groups with Lie algebras $\k \subset \g$, and these spaces are unique up to covering (\cite{Loos}). If the symmetric pair is symplectic, then there is a $G$-invariant symplectic form on $G/K$ which is given by $\om$ on $\p \cong T_{eK}G/K$. Moreover, it was shown in \cite{CGRS, KS} that for each extrinsic symplectic realization $d\imath: \g \ra \asp(V, \Om)$ there is a corresponding symmetric space $M := G/K$, a Lie group homomorphism $\imath: G \ra ASp(V, \Om)$ and an injective e.s.s.i. $\und \imath: M \ra V$ with $\und \imath(eK) = 0$ and $d\und \imath_{eK} (T_{eK} G/K) = V_1$ for which the diagram 

\[ \label{diagram}
\xymatrix{
\g \ar[r]^\exp \ar[d]^{d\imath} & G \ar[r] \ar[d]^{\imath} & M \ar[d]^{\und \imath}\\
\asp(V, \Om) \ar[r]^\exp & ASp(V, \Om) \ar[r] & V
}
\]
commutes, where the maps $G \ra M = G/K$ and $ASp(V, \Om) \ra V = ASp(V, \Om)/Sp(V, \Om)$ are the canonical projections.

If $d\imath: \g \ra \asp(V, \Om)$ is an extrinsic symplectic realization of the real Lie algebra $\g$ and the real symplectic vector space $(V, \Om)$, then clearly, the complexification $d\imath_\C: \g_\C \ra V_\C$ with $\g_\C := \g \ot \C$ and $V_\C := V \ot \C$ is also a symplectic realization, where $\Om_\C \in \Lambda^2 V_\C$ is the complexification of $\Om$. Restricting to the complex case has the advantage that the structure theory of Lie algebras behaves better than in the real case. Since the solvability of the transvection Lie algebra is invariant under complexification, we may for the proof of the main result restrict ourselves to the investigation of extrinsic symplectic realizations of {\em complex }Lie algebras on {\em complex} symplectic vector spaces, which we shall do for the remainder of this article.

\subsection{Irreducible representations of $\sl(2, \C)$}

A basis $\{e_0, e_+, e_-\}$ of $\sl(2, \C)$ is called a {\em standard basis} if the bracket relations
\be \label{eq:brackets-sl(2,C)}
{}[e_0, e_\pm] = \pm 2 e_\pm \mbox{ and }Ê[e_+, e_-] = e_0
\ee
hold. We recall that for each $n \in \N_0$, there is a unique $(n+1)$-dimensional irreducible representation of $\sl(2, \C)$ which we denote by $M_n$. Indeed, for a standard basis of $\sl(2, \C)$ there is a basis $\{m_{-\frac n2}, \ldots, m_{\frac n2}\}$ of $M_n$ such that  we have for $r = -\frac n2, \ldots, \frac n2$
\be \label{eq:standard-basis} \ba{lll}
e_0 \cdot m_r & = & 2 r\ m_r\\ e_+ \cdot m_r & = & (n/2-r)\ m_{r+1}\\ e_- \cdot m_r & = & (n/2+r)\ m_{r-1}\ea.
\ee
A basis of $M_n$ satisfying (\ref{eq:standard-basis}) will be called a {\em standard basis of $M_n$}. In particular, since any representation of $\sl(2, \C)$ is the direct sum of irreducible representations, it follows from this description that for any representation $\phi: \sl(2, \C) \ra End(V)$ we have for $E_0 := \phi(e_0), E_\pm := \phi(e_\pm) \in End(V)$:
\bi
\item
$E_0$ is diagonalizable with integer eigenvalues.
\item
If $v$ is an eigenvector of $E_0$ with eigenvalue $\la \geq 2k > 0$ ($\la \leq -2k < 0$, respectively) for $k \in \N$, then $v = (E_+)^k \cdot w$ ($w = (E_-)^k \cdot w$, respectively), and $w$ is an eigenvector of $e_0$ with eigenvalue $\la - 2k$ ($\la + 2k$, respectively).
\item
If $v \neq 0$ is an eigenvector of $E_0$ with eigenvalue $k > 0$ ($k < 0$, respectively), then $(E_-)^l \cdot v \neq 0$ ($(E_+)^l \cdot v \neq 0$, respectively) for $l = 1, \cdots, k$.
\item
If $v \in M_n$ is an eigenvector of $E_0$ with eigenvalue $\la$, then
\be \label{eq:multiply-e+e-}
\ba{lll}
E_+ E_- v & = \left(\frac n2 \left(\frac n2 + 1\right) - \la (\la - 1)\right) v\\ & & \mbox{ and }\\ E_- E_+ v & = \left(\frac n2 \left(\frac n2 + 1\right) - \la (\la + 1)\right) v.
\ea \ee
\ei

\subsection{Levi subalgebras}
Let $\g$ be a complex Lie algebra, and let $\r \subset \g$ be the {\em radical of $\g$}, i.e. the (unique) maximal solvable ideal of $\g$. Then $\g/\r$ is semi-simple (\cite{Bourbaki}). A {\em Levi subalgebra of $\g$} is a subalgebra $\l \subset \g$ isomorphic to $\g/\r$ such that
\[
\g = \l \oplus \r
\]
as a vector space. Levi subalgebras always exist, but they are in general not unique. If $d\sigma: \g \ra \g$ is an involution, then going through the proof of existence of Levi subalgebras in \cite{Bourbaki}, one can easily see that the Levi subalgebra may be chosen $d\sigma$-invariantly, so that we deduce the following statement.

\begin{prop}
Let $\g = \k \oplus \p$ be a complex symmetric pair. Then there is a Levi subalgebra $\l \subset \g$ such that $\l = (\l \cap \k) \oplus (\l \cap \p)$ and $\g = \l \oplus \r$. In particular,
\[
\k = (\l \cap \k) \oplus (\r \cap \k), \mbox{ and } \p = (\l \cap \p) \oplus (\r \cap \p).
\]
\end{prop}
%
We also observe that $\l = (\l \cap \k) \oplus (\l \cap \p)$ is transvective if $\g = \k \oplus \p$ is. 
%
%

\section{Rank one reductions} \label{sec:rank-one}

Let $\g$ be a complex Lie algebra. We say that $\g$ is {\em graded}, if there is a direct sum decomposition
\[
\g = \bigoplus_{k \in \Z} \g_k
\]
such that $[\g_k, \g_l] \subset \g_{k+l}$. A graded Lie algebra always admits a symmetric decomposition $\g = \k \oplus \p$ with
\[
\k = \g^{ev} \mbox{ and } \p = \g^{odd}, \mbox{ where } \g^{ev} = \bigoplus_{k \in \Z} \g_{2k} \mbox{ and } g^{odd} = \bigoplus_{k \in \Z} \g_{2k+1}.
\]

\begin{df}
A graded Lie algebra $\g = \bigoplus_{k \in \Z} \g_k$ is said to be {\em of rank-one type} if there are elements $0 \neq e_0 \in \g_0$ and $e_\pm \in \g_{\pm 1}$ such that
\bi
\item
$e_0$ is a grading element, i.e., $ad_{e_0}|_{\g_k} = 2 k Id_{\g_k}$ for all $k \in \Z$,
\item
$[e_+, e_-] = e_0$, so that $\{e_0, e_+, e_-\}$ satisfies (\ref{eq:brackets-sl(2,C)}) and hence is a standard basis of a Lie subalgebra $\l \subset \g$,
\item
$\l \cong \sl(2, \C)$ is a Levi algebra of $\g$.
\ei
\end{df}

In particular, since $ad_{e_0}$ has only even eigenvalues, it follows from (\ref{eq:standard-basis}) that in a graded Lie algebra of rank one type any $ad_\l$-irreducible subspace of $\g$ must be equivariantly isomorphic to $M_n$ with $n$ {\em even}.

\begin{df} \label{df:rank-one-reduction}
Let $\g = \l \oplus \r$ be a symplectic Lie algebra of rank one type. An extrinsic symplectic realization
\[
d\imath: \Lambda + v: \g \longrightarrow \asp(V, \Om) = \sp(V, \Om) \oplus V
\]
with $V = V_1 \oplus V_2$ is said to be of {\em rank one type} if
\bi
\item
$\Lambda(e_0) \in \sp(V, \Om)$ has only even eigenvalues,
\item
$V_i$ is the sum of the eigenspaces of $\Lambda(e_0)$ with eigenvalue $\cong 2i \mod 4$ for $i = 1,2$.
\ei
\end{df}

With this notion, we are now ready to state the main result of this section.

\begin{prop} \label{prop:reduce-to-rank1}
Let $(\g, d\sigma)$ be a transvective complex symmetric pair with the corresponding decomposition $\g =  \k \oplus \p$, and suppose that $\g$ is not solvable. Then the following hold.
\bi
\item
There is a $\sigma$-invariant transvective graded subalgebra $\g' = \bigoplus_{k \in \Z} \g'_k \subset \g$ of rank-one type such that $\g' \cap \k = (\g')^{ev}$ and $\g' \cap \p = (\g')^{odd}$.
\item
If $\g = \k \oplus \p$ is a symplectic symmetric pair, then so is $\g' = (\g')^{ev} \oplus (\g')^{odd}$.
\item
If $\g$ admits an extrinsic symplectic realization, then $\g'$ admits an extrinsic symplectic realization of rank one type.
\ei
\end{prop}


\begin{proof}
Let $\l \subset \g$ be a $\sigma$-invariant Levi algebra, and note that $\l \neq 0$ as $\g$ is assumed not to be solvable, hence we have the transvective symmetric pair $\l = (\l \cap \k) \oplus (\l \cap \p) =: \k_\l \oplus \p_\l$ with a {\em semi-simple }Lie algebra $\l$. It follows that $\k_\l$ is {\em reductive}, i.e., the direct sum of a semi-simple and an abelian Lie algebra (\cite{Hel}). Let $\t_0 \subset \k_\l$ be a Cartan subalgebra of $\k_\l$. Thus, regarding $\p$ and $\k$ as $\k_\l$-modules, we obtain the weight space decompositions
\be \label{eq:decompose-k/p}
\p = \bigoplus_{\la \in \Phi} \p^\la \mbox{ and } \k = \bigoplus_{\mu \in \Psi} \k^\mu
\ee
for subsets $\Phi, \Psi \subset \t_0^*$. Moreover, for $x \in \t_0 \subset \l$ both $\l$ and $\r$ are $ad_x$-invariant, so that these weight spaces decompose as
\be \label{eq:p-la splits}
\ba{lll}
& \p^\la = (\p^\la \cap \l) \oplus (\p^\la \cap \r) =: \p^\la_\l \oplus \p^\la_\r\\ \mbox{ and }\\Ê& \k^\mu = (\k^\mu \cap \l) \oplus (\k^\mu \cap \r) =: \k^\mu_\l \oplus \k^\mu_\r.
\ea \ee
From the bracket relation (\ref{eq:symmetric-brackets}) we read off the relations
\be \label{eq:brackets-k/p} \ba{llll}
[\k^\mu, \k^{\mu'}] \subset \k^{\mu + \mu'}, & [\k^\mu, \p^\la] \subset \p^{\mu + \la} & \mbox{and}Ê& [\p^\la, \p^{\la'}] \subset \k^{\la + \la'}
\ea \ee
for all $\la, \la' \in \Phi$ and $\mu, \mu' \in \Psi$. Since $\k_\l$ acts non-trivially on $\p_\l$, there is a $0 \neq \la_0 \in \Phi$ such that $\p^{\la_0}_\l \neq 0$. Moreover, we may assume w.l.o.g. that $\p^{k \la_0}_\l = 0$ for all $k$ with $|k| > 1$. We define
\[
\g'_{2k+1} := \p^{(2k+1) \la_0} \mbox{ and } \g'_{2k} := \bigoplus_{n+m=k} [\g'_{2n+1}, \g'_{2m-1}] \subset \k^{2k \la_0}.
\]
Then it follows from (\ref{eq:brackets-k/p}) that $\g' := \bigoplus_{k \in \Z} \g'_k \subset \g$ is a transvective graded Lie algebra with $\g' \cap \p = (\g')^{odd}$ and $\g' \cap \k = (\g')^{ev}$. Also, by construction, we have $\g' \cap \p_\l = \p_\l^{\la_0} \oplus \p_\l^{-\la_0} \subset \g_1' \oplus \g_{-1}'$, and as the Killing form $B_\l$ of $\l$ is non-degenerate and $B_\l(\p_\l^\la, \p_\l^\mu) = 0$ for $\la + \mu \neq 0$, it follows that $B_\l|_{\g' \cap \p_\l}$ is non-degenerate.

Since $B_\l|_{\k_\l}$ is non-degenerate and $\la_0 \in \t_0^*$ is in the weight lattice, it follows that there is an $e_0 \in \t_0$ such that
\[
\la_0(x) = 2 \frac{B_\l(e_0, x)}{B_\l(e_0, e_0)}
\]
for all $x \in \t_0$. Then for $x \in \t_0$ and $p_\pm \in \p_\l^{\pm \la_0}$ we have
\[
B_\l(x, [p_+, p_-]) = B_\l([x, p_+], p_-) = \la_0(x) B_\l(p_+, p_-) = 2 \frac{B_\l(p_+, p_-)}{B_\l(e_0, e_0)} B_\l(e_0, x),
\]
so that
\[
{}[p_+, p_-] = 2 \frac{B_\l(p_+, p_-)}{B_\l(e_0, e_0)} e_0,
\]
and hence, $\g_0' \cap \l = [\p_\l^{\la_0}, \p_\l^{-\la_0}] \subset span(e_0)$ is (at most) one dimensional. Moreover, as $B_\l|_{\g' \cap \p_\l}$ is non-degenerate, we can pick $e_\pm \in \p_\l^{\pm \la_0}$ with $B_\l(e_+, e_-) = \frac12 B(e_0, e_0)$, so that $\{e_0, e_+, e_-\}$ satisfies (\ref{eq:brackets-sl(2,C)}) and hence is a standard basis of a subalgebra $\l' \subset \g \cap \l$, and $\l' \cong \sl(2, \C)$. 

We decompose $\g'$ into $ad_{\l'}$-irreducible subspaces. Note that since $ad_{e_0}|_{\g'_r} = 2r Id$, so that $ad_{e_0}$ has only even eigenvalues, it follows that the $ad_{\l'}$-irreducible subspaces must be isomorphic to modules $M_n$ with $n$ {\em even}. On the other hand, for each such module $ad_{e_0}|_{M_n}$ has $0$ as an eigenvalue and hence each such summand intersects $\g_0'$. Thus, since $\dim \g'_0 \cap \l = 1$, it follows that $\g' \cap \l$ is $ad_{\l'}$-irreducible, hence we conclude that $\g' \cap \l = \l' \cong \sl(2, \C)$. Moreover, by construction we have
\[
\g' = (\g' \cap \l) \oplus (\g' \cap \r) = \l' \oplus (\g' \cap \r),
\]
and since $\l' \cong \sl(2, \C)$ is simple and $(\g' \cap \r) \lhd \g'$ is solvable, it follows that $\l'$ is a Levi algebra of $\g'$. Thus, $\g'$ is of rank one type.

In order to see the second statement, suppose that $\om \in \Lambda^2\p^*$ is an $ad_\k$-invariant symplectic form. Then $\om(\p^\la, \p^{\la'}) = 0$ whenever, $\la + \la' \neq 0$. Therefore, in this case we have $\Phi = -\Phi$, and the non-degenericity of $\om|_{(\g')^{odd}}$ follows immediately.

For the last claim, suppose that $d\imath = \Lambda + v: \g \ra \asp(V, \Om)$ is an extrinsic symplectic realization and $\Lambda(\k) \subset \tilde \k$, $\Lambda(\p) \subset \tilde \p$ with $\tilde \k, \tilde \p$ from (\ref{eq:define tilde k/p}) for some splitting $V = V_1 \oplus V_2$, and $V_1 = v(\p)$.

We assert that $\ker \Lambda \cap \l = 0$. Namely, if $k \in \ker \Lambda \cap \l$, then for all $x \in \p$ we have
\[
0 = \Lambda(k) \cdot v(x) = v([k, x]),
\]
and since $v$ is injective by hypothesis, it follows that $[k, \p] = 0$. But $\g$ is transvective and hence generated by $\p$, hence it follows that $k \in \z(\g) \cap \l = 0$, where the latter follows since $\l$ is semi-simple. 

Therefore, if we let $\tilde \l := \Lambda(\l) \subset \sp(V, \Om)$, then $\tilde \l \cong \l$, and thus, $\tilde \t_0 := \Lambda(\t_0) \subset \tilde \k$ acts diagonalizably on $V_1$ and $V_2$, so that we get again a weight space decomposition of $V_i$. Indeed, since $v: \p \ra V_1$ is an $(\k \cap \l)$-equivariant isomorphism, it follows that 
\[
V_1 = \bigoplus_{\la \in \Phi} V_1^\la, \mbox{ and } V_2 = \bigoplus_{\nu \in \Pi} V_2^\nu.
\]
where $\Phi \subset \t_0^* \cong \tilde \t_0^*$ equals the set of weights of $\p$ from (\ref{eq:decompose-k/p}) and $v(\p^\la) = V_1^\la$ for all $\la \in \Phi$, and for some $\Pi \subset \tilde \t_0^*$. Now consider the element $\la_0 \in \Phi$ from above, and define
\[
V' := V_1' \oplus V_2', \mbox{ where } V_1' := \bigoplus_{k \in \Z} V_1^{(2k+1) \la_0} \mbox{ and } V_2' := \bigoplus_{k \in \Z} V_2^{2k \la_0}.
\]
Since $\Om(V_1^\la, V_1^{\la'}) = 0$ and $\Om(V_2^\mu, V_2^{\mu'}) = 0$ whenever $\la + \la' \neq 0$ and $\mu + \mu' \neq 0$, it follows that $\Om|_{V'}$ is non-degenerate. Moreover, $V_1' = v((\g')^{odd})$. Finally, we have for $x \in \t_0$, $g_k \in \g'_k$ and $w_l \in V_i^{l \la_0}$
\[
\Lambda(x) \cdot \Lambda(g_k) \cdot (w_l) = \Lambda(\underbrace{[x, g_k]}_{=k \la_0(x) g_k}) \cdot w_l + \Lambda(g_k) \cdot \underbrace{\Lambda(x) \cdot w_l}_{= l \la_0(x) w_l} = (k + l) \la_0(x) \Lambda(g_k) \cdot w_l,
\]
so that $\Lambda(\g'_k) \cdot V_i^{l \la_0} \subset V_{i+k}^{(k + l) \la_0}$ for all $k \in \Z$, taking the indices $i, i+k$ mod $2$. Therefore, $V'$ is $\Lambda(\g')$-invariant.

Thus, the map $\g' \ra \asp(V', \Om)$, $x \mapsto \Lambda(x)|_{V'} + v(x)$ is an extrinsic symplectic realization, and since $\Lambda(e_0)$ acts as $2k Id$ on $V_i^{k \la_0}$, this realization is of rank one type.
\end{proof}

\section{Extrinsic realizations of rank one type} \label{sec:solvable}

In this section, we shall prove that there are no extrinsic realizations of rank one type which by 
Proposition~\ref{prop:reduce-to-rank1} implies our main result. As a first step, we show the following.

\begin{prop} \label{prop:z(rad)}
Let $\g = \l \oplus \r$ be a graded transvective symplectic Lie algebra of rank one type. Then there is an $ad_\l$-invariant abelian subalgebra $\a \subset \r$ such that
\bi
\item
$\a \cong M_2$ or $\a \cong M_4$ as an $\l$-module.
\item
$\a \cap \p = (\r \cap \p)^{\perp_\om}$.
\item
$[\a, \r] \subset \z(\g) \subset \g_0$.
\ei
In particular, $\r \cap \p \subset \p$ is coisotropic.
\end{prop}

\begin{proof}
First we observe that $[\g, \g] = \g$. Namely, $\p \subset \g^{odd}$ by hypothesis, so that $\p = [e_0, \p] \subset [\g, \g]$, and $\k = [\p, \p] \subset [\g, \g]$ as $\g$ is transvective. Thus, $\r \lhd \g$ is nilpotent (\cite{Bourbaki}).

Consider the Lie algebra $\tilde \g := \g/\z(\g)$. As $\z(\g) \subset \r$, we have $\tilde \g = \l \oplus \tilde \r$, where $\tilde \r = \r(\g)/\z(\g)$ is again nilpotent. Thus, $\z(\tilde \r) \neq 0$, and for the canonical projection $\pi: \g \ra \tilde \g$ we consider
\[
\hat \z := \pi^{-1}(\z(\tilde \r)) \lhd \g.
\]
Note that $\z(\g) \subsetneq \hat \z$, and $\hat \z$ is $ad_\l$-invariant, so that we get an $ad_\l$-invariant decomposition
\[
\hat \z = \a \oplus \z(\g)
\]
with $\a \neq 0$ and $[\r, \a] \subset \z(\g)$. Observe that $\z(\g) \subset \g_0 \subset \k$ since $[e_0, \z(\g)] = 0$. Moreover, by the $ad_\l$-invariance, we have $\a \cap \p = [e_+, \a \cap \k] + [e_-, \a \cap \k]$, so that
\[
\om(\a \cap \p, \r \cap \p) = \sum_{\eps = \pm 1}\om([e_\eps, \a \cap \k], \r \cap \p) = \sum_{\eps = \pm 1}\om(e_\eps, \underbrace{[\a \cap \k, \r \cap \p]}_{{\subset \z(\g) \cap \p = 0}}) = 0.
\]
Thus, $\a \cap \p \subset (\r \cap \p)^{\perp_\om}$ so that, in particular, $\dim(\a \cap \p) \leq 2$.

Next, we assert that $\a$ contains no non-zero $ad_\l$-invariant element. For if $a_0 \in \a$ would satisfy $[a_0, \l] = 0$, then $a_0 \in \g_0 \subset \k$ and $[a_0, \p \cap \r] \subset \p \cap \z(\g) = 0$. Since $\p = (\p \cap \l) \oplus (\p \cap \r)$, this implies $[a_0, \p] = 0$ and hence $a_0 \in \z(\g)$ as $\g$ is generated by $\p$. Thus, $a_0 \in \a \cap \z(\g) = 0$ as asserted.

Since all eigenvalues of $ad_{e_0}$ are even, it follows that as an $\l$-module, $\a \cong M_{2n_1} \oplus \ldots \oplus M_{2n_k}$ for integers $n_1, \ldots, n_k \geq 1$. Looking at the eigenvalues of $ad_{e_0}$ on $M_{2n_i}$, it follows that
\[
\dim (\a \cap \p) = 2 \left\lfloor \frac{n_1}2 \right\rfloor  + \ldots + 2 \left\lfloor \frac{n_k}2 \right\rfloor,
\]
and since $0 < \dim(\a \cap \p) \leq 2$, it follows that $\dim(\a \cap \p) = 2$ so that $\a \cap \p = (\r \cap \p)^{\perp_\om}$, and $\a \cong M_{2n}$ for $n \in \{1, 2\}$. Finally, $\a$ is abelian since $[\ ,\ ]: \Lambda^2 \a \ra \z(\g)$ must be $ad_\l$-equivariant. But by the Clebsch-Gordon formula, $\Lambda^2 M_{2n}$ does not contain an $ad_\l$-trivial summand, so that this bracket must vanish.
\end{proof}

Let $\{a_{-n}, \ldots, a_n\}$ be a standard basis of $\a \cong M_{2n}$ for $n \in \{1, 2\}$. Since $\om(a_+, e_+) = \om(a_-, e_-) = 0$ by the $ad_{e_0}$-invariance of $\om$ and $a_\pm \in (\p \cap \r)^{\perp_\om}$, we have $\om(e_+, a_-) \neq 0$ and $\om(e_-, a_+) \neq 0$. In fact, $\om(e_+, a_-) = \om(e_-, a_+)$ since by (\ref{eq:multiply-e+e-})
\[
n\ \om(e_-, a_+) = \om(e_-, [e_+, a_0]) = - \om([e_-, a_0], e_+) = - n\ \om(a_-, e_+) = n\ \om(e_+, a_-).
\]
Thus  -- after rescaling the basis $\{a_{-n}, \ldots, a_n\}$ appropriately -- we may assume that
\be \label{eq:scale-a_n}
\ba{lll}
\om(e_+, a_-) = \om(e_-, a_+) = n & \mbox{ and } & \om(a_+, a_-) = \om(e_+, a_+) = \om(e_-, a_-) = 0.
\ea
\ee

For a more convenient description of elements of $\sp(V, \Om)$, we define the equivariant isomorphism
\be \label{eq:identify-sp(V)}
\circ: S^2(V) \longrightarrow \sp(V, \Om), \mbox{\hspace{1cm}} (x \circ y) \cdot z := \Om(x, z) y + \Om(y, z) x.
\ee
Then the Lie bracket in $\sp(V, \Om)$ can be characterized by the identity
\[
{}[A, x \circ y] = (Ax) \circ y + x \circ (Ay).
\]
Now we investigate extrinsic symplectic realizations of rank one type.

\begin{lem}
Let $\g = \l \oplus \r$ be a symplectic symmetric space of rank one type, and let
\[
d\imath = \Lambda + v: \g \longrightarrow \asp(V, \Om) = \sp(V, \Om) \oplus V
\]
be an extrinsic symplectic realization of rank one type. Moreover, let $\a \subset \r$ be the abelian subalgebra from Proposition~\ref{prop:z(rad)} with a standard basis $\{a_{-n}, \ldots, a_n\}$ satisfying (\ref{eq:scale-a_n}). Then the following holds.
\bi
\item
$\Lambda|_\l$ is injective, so that $\Lambda(\l) \subset \sp(V, \Om)$ is isomorphic to $\sl(2, \C)$.
\item
Let $E_k := \Lambda(e_k)$ and $A_k := \Lambda(a_k)$. Then $\{E_0, E_+, E_-\}$ is a standard basis of $\Lambda(\l)$, and $\{ A_{-n}, \ldots, A_n\}$ is a standard basis of $\Lambda(\a)$ w.r.t. the adjoint action of $\Lambda(\l)$.
\item
If we use the symplectic isomorphism $v: \p \ra V_1$ to identify these spaces, we have
\be \label{eq:E_+e_-=E_-e_+}
\ba{llll}E_+ e_- = E_- e_+, & A_+ e_-  = E_- a_+ & \mbox{and}Ê& A_- e_+ = E_+ a_-.
\ea \ee
\item
With the identification $S^2(V) \cong \sp(V, \Om)$ from (\ref{eq:identify-sp(V)}) we have
\be \label{eq:form-A0}
\ba{ll} \left. \ba{l} A_0 = a_+ \circ a_- + \tilde A_0\\Ê\\ 
A_{\pm 2} = a_\pm \circ a_\pm + \tilde A_{\pm 2}, \mbox{if $n = 2$} \ea \right\} \mbox{ where $\tilde A_{2k} \in S^2(V_2)$}.
\ea \ee
In particular, $A_0 e_\pm = -n\ a_\pm$, and if $n = 2$ then $A_{\pm 2} e_\mp = - 4 a_\pm$, $A_{\pm 2} e_\pm = 0$.
\ei
\end{lem}

\begin{proof}
By assumtion, $\Lambda(e_0)|_{V_1}$ has only non-zero eigenvalues and hence is not zero as $V_1 \cong \p \neq 0$. Thus, $\Lambda|_\l \neq 0$, and since $\l \cong \sl(2, \C)$ is simple, it follows that $\Lambda|_\l$ is injective which shows the first claim, and the second follows since $\Lambda$ is a Lie algebra homomorphism.

For clarity of the proof, we shall make the use of the isomorphism $v: \p \ra V_1$ explicit. For the third property, observe that $\Lambda(x) v(y) = \Lambda(y) v(x)$ for all $x, y \in \p$, and applying this to $(x, y) = (e_+, e_-)$ and $(x, y) = (a_\pm, e_\mp)$ yields $E_+ e_- = E_- e_+$ and $A_\pm e_\mp = E_\mp a_\pm$.

For the last part, note that by (\ref{eq:multiply-e+e-})
\[ \ba{rl}
n (A_\pm + v(a_\pm)) =& d\imath(n a_\pm) = d\imath([e_\pm, a_0]) = [d\imath(e_\pm), d\imath(a_0)] = [E_\pm + v(e_\pm), A_0]\\ = & [E_\pm, A_0] - A_0 v(e_\pm),
\ea \]
so that $A_0 v(e_\pm) = - nÊ v(a_\pm)$. Furthermore, for $n = 2$ we have again by (\ref{eq:multiply-e+e-})
\[ \ba{rl}
4 (A_\pm + v(a_\pm)) =& d\imath(4 a_\pm) = d\imath([e_\mp, a_{\pm2}]) = [d\imath(e_\mp), d\imath(a_{\pm2})] = [E_\mp + v(e_\mp), A_{\pm2}]\\ = & [E_\mp, A_{\pm2}] - A_{\pm2} v(e_\mp),
\ea \]
so that $A_{\pm2} v(e_\mp) = -4 v(a_\pm)$. Finally,
\[ \ba{rl}
0 =& d\imath([e_\pm, a_{\pm2}]) = [d\imath(e_\pm), d\imath(a_{\pm2})] = [E_\pm + v(e_\pm), A_{\pm2}]\\ = & [E_\pm, A_{\pm2}] - A_{\pm2} v(e_\pm),
\ea \]
hence $A_{\pm2} v(e_\pm) = 0$. But if $p \in \p \cap \r$, then $[a_{2k}, p] \in \z(\g) \cap \p = 0$, and therefore for each such $p$ we have
\[ \ba{rl}
0 =& d\imath([a_{2k}, p]) = [d\imath(a_{2k}), d\imath(p)] = [A_{2k}, \Lambda(p) + v(p)] =  [A_{2k}, \Lambda(p)] + A_{2k} v(p),
\ea \]
so that $A_{2k} v(\p \cap \r) = 0$. Thus, $A_{2k} V_1 = A_{2k} v(\p) = span(A_{2k} v(e_\pm))$, and from here, the identification (\ref{eq:identify-sp(V)}) and the normailzation (\ref{eq:scale-a_n}) imply that $A_{2k}$ has the asserted form.
%
\end{proof}

Recall that $E_0$ has only even eigenvalues, so that we have a decomposition
\[
V = \bigoplus_{k \in \N_0} W_k, \mbox{ where } W_k \cong M_{2k} \ot \C^{l_k} \mbox{ as a $\tilde \l$-module}.
\]
Moreover, we decompose each $W_k$ into its $E_0$-eigenspaces
\[
W_k = \bigoplus_{l =-k, \cdots k} W_k^l \mbox{ with $E_0|_{W_k^l} = 2l Id$},
\]
so that
\begin{eqnarray*}
V_1 & = & \bigoplus \{ W_k^l \mid k \in \N_0, l \mbox{ odd, } |l| \leq k\},\\
V_2 & = & \bigoplus \{ W_k^l \mid k \in \N_0, l \mbox{ even, } |l| \leq k\}.
\end{eqnarray*}
We may thus decompose
\be \label{eq:decompose-a_pm}
a_\pm = \sum_{k \geq 1} a_\pm^k, \mbox{ where } a_\pm^k \in W_k^{\pm1}.
\ee
Moreover, we have the decompositions
\be \label{eq:decompose-S2(W)}
\ba{lll} & S^2(V_1) = \bigoplus \{ W_k^l \circ W_{k'}^{l'} \mid k, k' \in \N_0, l, l' \mbox{ odd, } |l| \leq k, |l'| \leq k'\}\\ \mbox{ and }\\Ê
& S^2(V_2) = \bigoplus \{ W_k^l \circ W_{k'}^{l'} \mid k, k' \in \N_0, l, l' \mbox{ even, } |l| \leq k, |l'| \leq k'\}
\ea \ee
In order to show that there is no extrinsic symplectic realization of $\g$, we shall exploit the following identities both of which follow immediately from (\ref{eq:multiply-e+e-}).

\begin{eqnarray}
\label{eq:condition-A_0} [E_+, [E_-, A_0]] & = & (n^2 + n) A_0,\\
\label{eq:condition-A_2} [E_\pm, [E_\pm, A_0]] & = & (n^2 - n) A_{\pm 2}.
\end{eqnarray}
The rest of the proof now will be split into several lemmas.

\begin{lem}
The endomorphism $\tilde A_0$ from (\ref{eq:form-A0}) has the form
\be \label{lem:form-A_0-tilde}
\tilde A_0 = \sum_{k, k' \geq 2} \al_{k, k'} (E_+ a_+^k) \circ (E_- a_-^{k'})
+ \sum_{k, k' \geq 1} \beta_{k, k'} (E_- a_+^k) \circ (E_+ a_-^{k'}) + \tilde A_{00},
\ee
where $\al_{k, k'} = (n^2 + n - k^2 - k - k'^2 - k' + 8)^{-1}$,  $\beta_{k, k'} = (n^2 + n - k^2 - k - k'^2 - k')^{-1}$, and $\tilde A_{00} \in W_0^0 \circ W_n^0$.
\end{lem}

\begin{proof} According to (\ref{eq:form-A0}) the $S^2(V_2)$-component of (\ref{eq:condition-A_0}) reads
\be \label{eq:lem-decompose}
(E_+ a_+) \circ (E_- a_-) + (E_- a_+) \circ (E_+ a_-) +  ([E_+, [E_-, \tilde A_0]])_{S^2(V_2)} = (n^2 + n) \tilde A_0.
\ee
Since $\tilde A_0 \in S^2(V_2)$ commutes with $E_0$, have
\[
\tilde A_0 \in \bigoplus \{ W_k^l \circ W_{k'}^{-l} \mid k, k' \in \N_0, l \geq 0 \mbox{ even, } l \leq k, k'\} =: S^2(V_2)_0,
\]
and this is the eigenspace decomposition of $S^2(V_2)_0$ of the endomorphim $pr_{S^2(V_2)}\circ ad_{E_+} \circ ad_{E_-}: S^2(V_2)_0 \ra S^2(V_2)_0$; indeed, for $v_k^l \circ v_{k'}^{-l} \in W_k^l \circ W_{k'}^{-l}$ with $l$ even we have by (\ref{eq:multiply-e+e-})
\[
([[E_+, [E_-, v_k^l \circ v_{k'}^{-l}]])_{S^2(V_2)} = (E_+ E_- v_k^l) \circ v_{k'}^{-l} + v_k^l \circ (E_+ E_- v_{k'}^{-l}) = (k^2 + k + k'^2 + k' - 2 l^2) v_k^l \circ v_{k'}^{-l}.
\]
Observe that for $n \in \{1, 2\}$ and $l$ even with $0 \leq l \leq k, k'$ we have $k^2 + k + k'^2 + k' - 2 l^2 = n^2 + n$ iff $l = 0$ and $\{k, k'\} = \{0, n\}$. Thus, if we decompose (\ref{eq:lem-decompose}) into its $W_k^l \circ W_{k'}^{-l}$-components, then we conclude that
\bi
\item For $k, k' \geq l \geq 4$ we have $(\tilde A_0)_{W_k^l \circ W_{k'}^{-l}} = 0$.
\item For $k, k' \geq 2$ we have $(E_+ a_+^k) \circ (E_- a_-^{k'}) + (k^2 + k + k'^2 + k' - 8) (\tilde A_0)_{W_k^2 \circ W_{k'}^{-2}} =$ \linebreak $(n^2 + n) (\tilde A_0)_{W_k^2 \circ W_{k'}^{-2}}$.
\item For $k \geq 1$ we have $(E_- a_+^k) \circ (E_+ a_-^{k}) + 2 (k^2 + k) (\tilde A_0)_{W_k^0 \circ W_{k'}^0} = (n^2 + n) (\tilde A_0)_{W_k^0 \circ W_{k'}^0}$.
\item For $k \neq k' \geq 1$ we have $(E_- a_+^k) \circ (E_+ a_-^{k'}) + (E_- a_+^{k'}) \circ (E_+ a_-^k) + (k^2 + k + k'^2 + k') (\tilde A_0)_{W_k^0 \circ W_{k'}^0} = (n^2 + n) (\tilde A_0)_{W_k^0 \circ W_{k'}^0}$.
\item
For $k \neq n$ we have $(\tilde A_0)_{W_0^0 \circ W_k^0} = 0$.
\ei
From this, (\ref{lem:form-A_0-tilde}) follows.
\end{proof}

\begin{lem} \label{lem:r-pm}
There are constants $r_\pm^k$ with $r_+^k r_-^k = 1$ such that
\[
E_+ a_-^k = r_+^k E_- a_+^k \mbox{ and } E_- a_+^k = r_-^k E_+ a_-^k.
\]
In particular, $a_+^k = 0$ iff $a_-^k = 0$.
\end{lem}

\begin{proof} Let us take the $(W_k^1 \circ W_{k'}^{-1})$-component of (\ref{eq:condition-A_0}) which by (\ref{eq:form-A0}) and (\ref{lem:form-A_0-tilde}) reads
\be \label{eq:W1-W1}
\ba{rll}
(E_+ E_- a_+^k) \circ a_-^{k'} + a_+^k \circ (E_+ E_- a_-^{k'}) + \al_{k,k'} (E_- E_+ a_+^k) \circ (E_+ E_- a_-^{k'}) & \\ \\
 + \beta_{k,k'} ((E_+ E_- a_+^k) \circ (E_- E_+ a_-^{k'}) + (E_-^2 a_+^{k'}) \circ (E_+^2 a_-^k)) & = (n^2 + n) a_+^k \circ a_-^{k'}.
\ea \ee
Now suppose that $a_+^k = 0$. Then (\ref{eq:W1-W1}) reads
\[
\beta_{k,k'} (E_-^2 a_+^{k'}) \circ (E_+^2 a_-^k) = 0.
\]
Pick $k'$ such that $a_+^{k'} \neq 0$ so that by (\ref{eq:standard-basis}) $E_-^2 a_+^{k'} \neq 0$. Since $\beta_{k,k'} \neq 0$, we conclude that $E_+^2 a_-^k = 0$ and hence, again by (\ref{eq:standard-basis}), $a_-^k = 0$. Similarly, we conclude that whenever $a_-^{k'} = 0$ then $a_+^{k'} = 0$, so that
\[
a_+^k = 0 \Longleftrightarrow a_-^k = 0.
\]

Since $E_\pm E_\mp a_+^k$ and $E_\pm E_\mp a_-^{k'}$ are multiples of $a_+^k$ and $a_-^{k'}$, respectively, and since $\beta_{k, k'} \neq 0$, it follows from (\ref{eq:W1-W1}) that $(E_-^2 a_+^{k'}) \circ (E_+^2 a_-^k)$ is a multiple of $a_+^k \circ a_-^{k'}$. Thus, if we pick $k = k'$ such that $a_\pm^k \neq 0$, then it follows that $E_-^2 a_+^k$ and $E_+^2 a_-^k$ must be multiples of $a_-^k$ and $a_+^k$, respectively. Thus, $(k^2 + k) E_- a_+^k = E_+ E_-^2 a_+^k$ is a multiple of $E_+ a_-^k$, i.e., $E_- a_+^k = r_-^k E_+ a_-^k$ for some $r_-^k \neq 0$, and the other equation follows with $r_+^k := (r_-^k)^{-1}$.
\end{proof}

\begin{lem} \label{lem:a_k, ge3}
We have $a_\pm^k = 0$ for all $k \geq 3$.
\end{lem}

\begin{proof} Let us take the $(W_k^3 \circ W_{k}^{-3})$-component of (\ref{eq:condition-A_0}). It reads
\[
\al_{k,k} (E_+^2 a_+^k) \circ (E_-^2 a_-^k) = 0,
\]
and since $\al_{k,k} \neq 0$, this implies that $E_+^2 a_+^k = 0$ or $E_-^2 a_-^k = 0$. For $k \geq 3$, this is equivalent to $a_+^k = 0$ or $a_-^k = 0$ by (\ref{eq:standard-basis}), and with the equivalence of these two equations from Lemma~\ref{lem:r-pm}, the claim follows.
\end{proof}

\begin{lem} \label{n=1}
We must have $n = 2$.
\end{lem}

\begin{proof}
We calculate the $S^2(V_1)$-component of (\ref{eq:condition-A_0}). First of all, we have by (\ref{eq:multiply-e+e-})
\[ \ba{lll}
{}[E_+, [E_-, a_+ \circ a_-]]_{S^2(V_1)} & = & (E_+ E_- a_+) \circ a_- + a_+ \circ (E_+ E_- a_-)\\
& = & \sum_{k, k' \in \{1,2 \}} (k^2 + k + k'^2 + k' - 2) a_+^k \circ a_-^{k'}.
\ea \]
By Lemma~\ref{lem:a_k, ge3}, we can simplify the form of $\tilde A_0$ from (\ref{lem:form-A_0-tilde}) to
\[
\tilde A_0 = \al_{2,2} (E_+ a_+^2) \circ (E_- a_-^2) + \sum_{k, k' \in \{1,2\}} \beta_{k,k'} (E_- a_+^k) \circ (E_+ a_-^{k'}) + \tilde A_{00}
\]
Thus,
\[ \ba{lll}
{}[E_+, [E_-, \tilde A_0]]_{S^2(V_1)} & = & \al_{2,2} (\underbrace{E_- E_+ a_+^2}_{=6 a_+^2}) \circ (\underbrace{E_+ E_- a_-^2}_{=4 a_-^2})\\
& & + \sum_{k, k' \in \{1,2\}} \beta_{k,k'} (\underbrace{E_+ E_- a_+^k}_{=(k^2 + k) a_+^k}) \circ (\underbrace{E_- E_+ a_-^{k'}}_{=(k'^2 + k') a_+^{k'}})\\
& & \mbox{\hspace{2cm}} + \beta_{k,k'} (\underbrace{E_-^2 a_+^k}_{= r_-^k (k^2 + k) a_-^k}) \circ (\underbrace{E_+^2 a_-^{k'}}_{= r_+^{k'} (k'^2 + k') a_+^{k'}})\\
& = & 24 \al_{2,2} a_+^2 \circ a_-^2
+ \displaystyle{\sum_{k, k' \in \{1,2\}}} \beta_{k,k'} (k^2 + k)(k'^2 + k')(1+r_-^{k'}r_+^k) a_+^k \circ a_-^{k'}.
\ea \]
Thus, using that $r_+^k r_-^k = 1$, the $(W_k^1 \circ W_k^{-1})$-components of (\ref{eq:condition-A_0}) read
\[ \ba{lll}
k = 1: & (2 + 8 \beta_{1,1}) a_+^1 \circ a_-^1 = (n^2 + n) a_+^1 \circ a_-^1\\
k = 2: &
(10 + 24 \al_{2,2} + 72 \beta_{2,2}) a_+^2 \circ a_-^2 = (n^2 + n) a_+^2 \circ a_-^2
\ea \]
If $n=1$, then $2 + 8 \beta_{1,1} \neq (n^2 + n)$ and $(10 + 24 \al_{2,2} + 72 \beta_{2,2}) \neq (n^2 + n)$, so that these equations would imply that $a_+^1 \circ a_-^1 = a_+^2 \circ a_-^2 = 0$ which together with Lemmas~\ref{lem:r-pm} and \ref{lem:a_k, ge3} whould imply that $a_\pm^k = 0$ for all $k$, i.e., $a_\pm = 0$ which is impossible.
%
\end{proof}

\begin{lem} \label{lem:gamma}
Let $n = 2$, and let $k \in \{1,2\}$ be such that $a_\pm ^k \neq 0$. Then
\[
r_\pm^k (k^2 + k) + 2 \beta_{k,k} r_\pm^k (k^2 + k)^2 = 2.
\]
and hence 
\[
r_\pm^k = \frac2{(k^2 + k) + 2 \beta_{k,k} (k^2 + k)^2} = \left\{ \ba{rl} 1/3 & \mbox{if $k = 1$}\\ -1/3 & \mbox{if $k = 2$} \ea \right.
\]
\end{lem}

Evidently, this is the final step in our proof, since $|r_\pm^k| = 1/3$ contradicts the identity $r_+^k r_-^k = 1$ from Lemma~\ref{lem:r-pm}.

\begin{proof}
We only show the lemma for ``$\pm$'' $=$ ``$+$'' as the other case goes through analogously. By (\ref{eq:form-A0}), the $(W_k^1 \circ W_k^1)$-component of (\ref{eq:condition-A_2}) reads in this case
\be \label{eq:lem-form+}
a_+^k \circ (E_+^2 a_-^k) + ([E_+, [E_+, \tilde A_0]])_{W_k^1 \circ W_k^1} = 2 a_+^k \circ a_+^k.
\ee
By Lemma~\ref{lem:r-pm} and (\ref{eq:multiply-e+e-}) we have
\be \label{eq:a+-E+2a-}
E_+^2 a_-^k = r_+^k E_+ E_- a_+^k = r_+^k (k^2 + k) a_+^k.
\ee
By Lemma~\ref{lem:a_k, ge3}, we can simplify the form of $\tilde A_0$ from (\ref{lem:form-A_0-tilde}) to
\[
\tilde A_0 = \al_{2,2} (E_+ a_+^2) \circ (E_- a_-^2) + \sum_{k, k' \in \{1,2\}} \beta_{k,k'} r_+^{k'} (E_- a_+^k) \circ (E_- a_+^{k'}) + \tilde A_{00}
\]
and thus,
\be \label{eq:S2-V1}
\ba{lll}
([E_+, [E_+, \tilde A_0]])_{W_k^1 \circ W_k^1} & = & 2 \beta_{k,k} r_+^k (\underbrace{E_+ E_- a_+^k}_{(k^2 + k) a_+^k}) \circ (\underbrace{E_+ E_- a_+^k}_{(k^2 + k) a_+^k})\\ \\
& = & 2 \beta_{k,k} r_+^k (k^2 + k)^2 a_+^k \circ a_+^k.
\ea
\ee
Substituting (\ref{eq:a+-E+2a-}) and (\ref{eq:S2-V1}) into (\ref{eq:lem-form+}) yields
\[
(r_+^k (k^2 + k) + 2 \beta_{k,k} r_+^k (k^2 + k)^2) a_+^k \circ a_+^k = 2 a_+^k \circ a_+^k,
\]
and since $a_+^k \neq 0$ by assumtion, the claim follows.
\end{proof}

\noindent
{\small {\sc Fakult\"at f\"ur Mathematik, Technische Universit\"at Dortmund, Vo\-gel\-poths\-weg 87, 
44221 Dortmund, Germany}}

\end{document}